\documentclass[11pt,a4paper]{amsart}
\usepackage{amsmath, amssymb, amscd, amsthm, graphicx,epsfig, color}

\newtheorem{theorem}{Theorem}
\newtheorem*{theorem*}{Theorem}
\newtheorem{lemma}[theorem]{Lemma}

\newtheorem{letteredtheorem}{Theorem}

\theoremstyle{definition}

\theoremstyle{remark}

\newcommand{\Chat}{\widehat{\mathbb{C}}}

\begin{document}

\title{Intrinsic circle domains}
\author{Edward Crane}
\date{\today}
\address{Heilbronn Institute for Mathematical Research\\
School of Mathematics\\
University of Bristol \\
BS8 1TW\\
United Kingdom}
\email{edward.crane@bristol.ac.uk}




\begin{abstract}
 Using quasiconformal mappings, we prove that any Riemann surface of finite
 connectivity and finite genus is  conformally equivalent to an intrinsic
 circle domain $\Omega$ in a compact Riemann surface $S$. This means that each
 connected component $B$ of $S\setminus \Omega$  is either a point or a closed
 geometric disc with respect to the complete constant curvature conformal
 metric of the Riemann surface $(\Omega \cup B)$. Moreover the pair $(\Omega,
 S)$ is unique up to conformal isomorphisms. We give a generalization to
 countably infinite connectivity. Finally we show how one can compute numerical approximations to intrinsic circle domains using circle packings and conformal welding.
\end{abstract}

\maketitle

\begin{center}\emph{Accepted for publication in Conformal Geometry and
    Dynamics}\end{center}

\section{Introduction}
Let $\Omega$ be a finitely connected domain in the Riemann sphere $\Chat$. A classical theorem of Koebe states that $\Omega$ is conformally equivalent to the
complement of a finite set of pairwise disjoint closed discs and points. Such
a domain is called a finitely connected \emph{circle domain}. Koebe's theorem
was extended by He and Schramm \cite{HS} to apply to domains with countably
many complementary components. Schramm later gave a different proof of
this result using transboundary extremal length \cite{S}. 
\begin{letteredtheorem}{\cite[Theorem 0.1]{HS}}\label{T: HS1}
Let $\Omega$ be a domain in $\Chat$ such that boundary $\partial\Omega$ has at
most countably many components. Then $\Omega$ is conformally homeomorphic to a circle domain $\Omega^{*}$ in $\Chat$. Moreover, $\Omega^{*}$ is unique up to M\"{o}bius transformations and every conformal automorphism of $\Omega^{*}$ is the restriction of a M\"{o}bius transformation.
\end{letteredtheorem}
If $\Omega$ is a domain in $\Chat$ that has at least three complementary
components, or at least one complementary component that is not a puncture,
then we say that $\Omega$ is \emph{hyperbolic}. The reason is that there is
then an unbranched analytic covering map from the unit disc $\mathbb{D}$ onto
$\Omega$, which can be used to transfer the Poincar\'{e} metric on
$\mathbb{D}$ to a complete conformal metric of constant curvature $-1$ on
$\Omega$, which is called the hyperbolic metric. 

Every ring domain is conformally equivalent to a round annulus, a punctured
disc, or the punctured plane. For a round annulus $\Chat \setminus
(B_1 \cup B_2)$, each complementary component $B_i$ is a spherical disc, i.e.~a closed ball in
the spherical metric. Observe that $B_i$ is
also a closed ball with respect to the hyperbolic metric belonging to the
domain $\Omega \cup B_i$, which is a larger spherical disc. In this paper we
generalize this property to obtain a new canonical form for multiply connected domains.  


\begin{theorem}\label{T: main}
Let $\Omega$ be a finitely connected domain in the Riemann sphere $\Chat$ with
complementary components $K_1, \dots, K_n$. Suppose that each domain $\Omega \cup
K_i$ is hyperbolic. Then $\Omega$ is conformally equivalent to a domain $\Omega^*$ in $\Chat$ with complementary components $L_1, \dots, L_n$ such that for each
$i=1, \dots, n$, $\Omega \cup L_i$ is hyperbolic and either $L_i$ is a puncture or $L_i$ is a closed disc with respect to the hyperbolic metric of $\Omega \cup L_i$. Moreover, $\Omega^*$ is unique up to M\"{o}bius transformations.
\end{theorem}

\noindent We call the canonical domains given by this theorem \emph{intrinsic circle domains}.

Theorem~\ref{T: main} has a generalization to positive genus ambient surfaces
in place of $\Chat$, which we will now explain. In any Riemann surface $S$,
 a \emph{closed geometric disc} will mean a closed
ball of some positive radius with respect to the appropriate complete
conformal metric of constant curvature on $S$, (i.e. hyperbolic, Euclidean or
spherical), with the extra condition that it must be homeomorphic to the
closed unit disc. Equivalently, the radius of the ball must be strictly less than
the injectivity radius of the metric at the center of the disc.  A circle
domain $\Omega^*$ in $S$ is a connected open subset of $S$ for which each
complementary component is either a point or a closed geometric disc in
$S$. Via the uniformization theorem, He and Schramm extended Theorem~\ref{T: HS1} to this setting, as follows.
\begin{letteredtheorem}{\cite[Theorem 0.2]{HS}}\label{T: HS2}
Let $\Omega$ be an open Riemann surface with finite genus and at most
countably many ends. Then there is a closed Riemann surface $R^*$ such that
$\Omega$ is conformally homeomorphic to a circle domain $\Omega^*$ in
$R^*$. Moreover the pair $(R^*, \Omega^*)$ is unique up to conformal homeomorphisms.
\end{letteredtheorem}

 Likewise, we can extend Theorem~\ref{T: main} to deal with the case of
 arbitrary finite genus. For the moment we also restrict ourselves to finite
 connectivity; we will relax this condition in \S\ref{S: countable}.
\begin{theorem}\label{T: general}
Let $\Omega$ be a Riemann surface of finite genus and finite
connectivity. Then there is a conformal embedding $\varphi$ of $\Omega$ into a
compact Riemann surface $S$, of the same genus as $\Omega$, so that $S
\setminus \varphi(\Omega)$ is the union of disjoint closed, connected and
simply connected sets $L_1, \dots, L_n$, and for each $i \in \{1, \dots, n\}$, $L_i$ is either a single point or a closed geometric disc with respect to the Riemann surface $\varphi(\Omega) \cup L_i$. Moreover, the pair $(S, \varphi)$ is unique up to conformal homeomorphisms.
\end{theorem}

The condition that $\Omega$ and $S$ have the same genus means that $S$ has the
minimal possible genus among all compact Riemann surfaces into which $\Omega$
may be embedded.  We will refer to the domain $\varphi(\Omega)$ as an \emph{intrinsic circle domain} in $S$. Note that Theorem~\ref{T: general} includes Theorem~\ref{T: main} as the special case of genus $0$, so this terminology is consistent; we shall only give a proof of Theorem~\ref{T: general}.

 There are some simple special cases. For example, there is precisely
 one case in which the appropriate geometry of $\varphi(\Omega) \cup L_i$ is
 spherical; this occurs when $\Omega$ is simply connected, in which case the
 statement reduces to the Riemann mapping theorem. The doubly connected genus
 $0$ cases occur when $\Omega$ is conformally equivalent to the punctured
 plane $\mathbb{C}^*$, the punctured disc $\mathbb{D}^*$, or a round
 annulus. Note that in the case of the punctured disc, the appropriate
 geometry for $\varphi(\Omega) \cup L_i$ is Euclidean for one complementary component and hyperbolic for the
 other. Another case in which more than one type of geometry must be
 considered occurs when $\Omega$ is a simply connected domain with two points
 removed; then the resulting canonical domain is either a triply-punctured
 sphere or it is $\mathbb{C}^*\setminus K$, where $K$ is the image under the
 exponential map of a closed disc of radius less than $\pi$. If $\Omega$ is a
 Riemann surface of genus $1$ with one end, then $\Omega^*$ is the complement
 of a point or of a Euclidean disc in $\mathbb{C}/\Lambda$, where $\Lambda$ is a lattice in $\mathbb{C}$. In all other cases the natural geometry of $\varphi(\Omega) \cup L_i$ is hyperbolic for every complementary component $L_i$.

Our definition of intrinsic circle domains was motivated by an observation
about extremal cases in the Gr\"otzsch or P\'olya-Chebotarev problem, which asks for the minimizer of the logarithmic capacity among compact
connected sets containing a given finite set of points in $\mathbb{C}$. This
problem often appears in the course of studying other extremal problems in
geometric function theory. For example, there are recent applications to the
Bloch-Landau constant \cite{CO} and to Smale's mean value conjecture \cite{Cr}. It is a
classical result of Lavrentiev and Goluzin that the extremal continuum $E$ is
unique and is the union of finitely many analytic arcs $A_i$, which are trajectories
of a certain rational quadratic differential (see for example \cite{Gol}). It
is not hard to show that each arc $A_i$ is a geodesic arc in the
hyperbolic metric of the domain $(\Chat \setminus E) \cup A_i$.
A similar condition arises for local minima of the condenser capacity for
domains separating one finite set of points from another.  In \cite{OP} it is
observed that the extremal continuum $E$ enjoys \emph{harmonic symmetry}. This
means that for any subarc $I$ of $E$, the harmonic measure of the
complementary domain $\Chat \setminus E$ with respect to the point $\infty$
assigns equal masses to each side of $I$. Note that both the geodesic arc
condition and the harmonic symmetry condition have the property that they may
be verified by checking each arc $A_i$ separately. To check $A_i$, we need to
know only the conformal class of the pair $(\Chat \setminus E, A_i)$. Our notion of intrinsic circle domains arose by analogy with this property.

We now outline the rest of the paper. Section~\ref{S: geometric lemma} gives a
simple qualitative distortion bound for quasiconformal extensions of conformal
maps between ring domains, which we later use several times. Section~\ref{S: uniqueness} gives
the proof of uniqueness in Theorem~\ref{T: general}, and section~\ref{S:
  existence} gives the existence proof. In section~\ref{S: countable}, we
extend Theorem~\ref{T: general} to include some cases with countably
infinitely many complementary components, subject to a geometric constraint. Section~\ref{S: mixed} deals
with a mixed condition, in which some boundary components are required to be
 circles in the spherical metric while the others are required to be intrinsic
 circles in the sense of Theorem~\ref{T: main}. In section~\ref{S: pictures}
 we discuss the use of circle packings for the numerical approximation of
 intrinsic circle domains, and illustrate with some examples.

\section{A geometric lemma}\label{S: geometric lemma}

\begin{lemma}\label{L: distortion bound}
Suppose $A_1$ and $A_2$ are ring domains in $\mathbb{C}$ whose inner boundaries
are circles $C_1$ and $C_2$ respectively. Suppose there is a conformal
homeomorphism $F: A_1 \to A_2$ under which $C_1$ corresponds to $C_2$. Then
the induced homeomorphism $F: C_1 \to C_2$ has a $K$-quasiconformal extension
between the interior discs of $C_1$ and $C_2$, where the constant $K$ depends
only on the conformal modulus of $A_1$. 
\end{lemma}

\begin{proof}
Lemma~\ref{L: distortion bound} is a consequence of \cite[Thm 1.4]{HS}, but it can also be proved
simply as follows. After applying similarities, we may assume that $C_1$ and
$C_2$ are both the unit circle $\partial \mathbb{D}$. Let $m$ be the conformal
modulus of $A_1$. Apply Schwarz reflection
across $\partial \mathbb{D}$ to extend $F$ to a
conformal homeomorphism $\tilde{F}$ between ring domains $\tilde{A}_1$ and
$\tilde{A}_2$. If $m = \infty$ then $\tilde{F}$ is a
conformal homeomorphism from the punctured plane to itself, so is in fact a
similarity, and its restriction to $\partial\mathbb{D}$ is an isometry. Otherwise, the doubled ring domains $\tilde{A}_i$ are hyperbolic, $\tilde{F}$ is a
hyperbolic isometry, and the hyperbolic length of $\partial \mathbb{D}$ in
each domain depends only on $m$. We claim that the density of
the hyperbolic metric of $\tilde{A}_1$ on $\partial \mathbb{D}$ is bounded
above and below in terms of $m$. Indeed, the hyperbolic length
of $\partial \mathbb{D}$ is a function of $m$ so we get such a
bound by applying the Koebe distortion theorem to a single-valued lift of
$\log \tilde{F}$ to the universal cover of $\tilde{A}_1$.

It follows that $F: \partial\mathbb{D} \to \partial\mathbb{D}$ is bi-Lipschitz , with constants that depend only on $m$. Therefore the
radial interpolation of the boundary correspondence provides a quasiconformal
extension with dilatation bounded in terms of the modulus of $A_1$. 
\end{proof}

\section{Uniqueness}\label{S: uniqueness}

Suppose that $\Omega$ and $\Omega'$ are two intrinsic circle domains in closed Riemann surfaces $S$ and $S'$ respectively. Let $f: \Omega \to \Omega'$ be a conformal homeomorphism. We have to show that $f$ is the restriction of a conformal isomorphism $\tilde{f}: S \to S'$. 

First, suppose that $\Omega$ is simply connected or doubly connected. If
$\Omega$ is simply connected or is doubly connected with infinite conformal modulus, then $S$
is the Riemann sphere and $\Omega$ is a circle domain. $\Omega'$ has the same
connectivity as $\Omega$ since they are homeomorphic, so $\Omega'$ is also a
circle domain in the Riemann sphere $S$. The uniqueness part of Koebe's
theorem then shows that $f$ is the restriction of a M\"{o}bius map.  Otherwise, $\Omega$ is a ring domain of finite modulus, as is $\Omega'$. Then
$S$ and $S'$ are both of genus $0$ and we have to show that $f$ extends to a
M\"{o}bius map. Consider a complementary component $B_1$ of $\Omega$, and let
$B_1'$ be the corresponding complementary component of $\Omega'$. Let $\pi:
\mathbb{D} \to \Omega \cup B_1$ and $\pi': \mathbb{D} \to \Omega' \cup B_i'$
be any Riemann maps. Then by hypothesis $\pi^{-1}(B_1)$ and $\pi'^{-1}(B_1')$
are closed hyperbolic discs in $\mathbb{D}$, which are bounded by Euclidean
circles. Thus $\pi^{-1}(\Omega)$ and $\pi'^{-1}(\Omega')$ are round
annuli. Any conformal homeomorphism between round annuli is the restriction
of a M\"{o}bius map, so we can extend $\pi'^{-1} \circ f \circ \pi$ to a
M\"{o}bius map $g$. Now define $f_1 = \pi' \circ g \circ \pi^{-1}$. Then
$f_1: \Omega \cup B_1 \to \Omega' \cup B_1'$  is a conformal homeomorphism
extending $f$. Similarly, $f$ extends to a conformal homeomorphism $f_2:
\Omega \cup B_2 \to \Omega \cup B_2'$. Gluing $f_1$ and $f_2$ together we
obtain a conformal homeomorphism from the Riemann sphere to itself which
extends $f$. This must be a M\"{o}bius map. Since $\Omega$ is conformally
equivalent to some round annulus, which is an intrinsic circle domain, this
in fact shows that $\Omega$ is a round annulus. 

Now suppose that $\Omega$ is at least triply connected. Consider any complementary component $B_i$ of $\Omega$, and let $B_i'$ be the
corresponding complementary component of $\Omega'$. This makes sense since
the homeomorphism $f$ induces a bijection between the ends of $\Omega$ and the
ends of $\Omega'$; for an intrinsic circle domain each end corresponds to
precisely one complementary component since the complementary components are all contractible.

Since $\Omega$ and $\Omega'$ are at least triply connected, $\Omega \cup B_i$
and $\Omega' \cup B_i'$ are not simply connected. Therefore we must consider
their universal covers, $\mathcal{U}$ and $\mathcal{U}'$ respectively, in order to
understand their hyperbolic metrics.  Let $\pi:
\mathcal{U} \to \Omega \cup B_i$  and $\pi': \mathcal{U}' \to \Omega' \cup
B_i'$ be unbranched analytic covering maps, with deck
transformation groups $\Gamma_i$ and $\Gamma_i'$
respectively. Since $B_i$ and $B_i'$ are contractible subsets of $\Omega \cup
B_i$ and $\Omega' \cup B_i'$ respectively, the homeomorphism $f$ induces a
homotopy equivalence $\Omega \cup B_i \to \Omega' \cup B_i'$, which in turn induces an
isomorphism $\rho_i: \Gamma_i \to \Gamma_i'$. Then $f$ lifts to a conformal
homeomorphism \[\hat{f}_i:\mathcal{U} \setminus \pi^{-1}\left(B_i\right) \to
\mathcal{U}' \setminus \pi'^{-1}\left(B_i'\right)\,.\] Note that $\hat{f}_i$ is $(\Gamma_i,
\Gamma_i')$-equivariant: for any element $\gamma \in \Gamma_i$ we have
\[ \rho_i(\gamma)\circ \hat{f}_i = \hat{f}_i \circ \gamma\,.\]
 The connected components of $\pi^{-1}\left(B_i\right)$
and of $\pi'^{-1}\left(B_i'\right)$ are disjoint closed discs because $\Omega$
and $\Omega'$ are intrinsic circle domains. Therefore $\hat{f}_i$ is a
conformal homeomorphism between circle domains. Hence by the uniqueness part of
Theorem~\ref{T: HS2} it is the restriction of a M\"{o}bius map $M_i$.  We find
that $M_i$ takes $\mathcal{U}$ onto $\mathcal{U}'$ since each of the two
circle domains has only one non-isolated boundary component. In particular,
$\mathcal{U} = \mathcal{U}'$. Moreover, the map $M_i: \mathcal{U} \to
\mathcal{U'}$ is also $(\Gamma_i, \Gamma_i')$-equivariant, since $\rho_i(\gamma)
\circ M_i \circ \gamma^{-1}: \mathcal{U} \to \mathcal{U'}$ is a M\"{o}bius map
extending $\hat{f}_i$ and is therefore equal to $M_i$. It follows that $M_i$
descends to a conformal homeomorphism $\tilde{f}_i : \Omega \cup B_i \to \Omega' \cup
B_i'$ that extends $f$. 

Gluing the extensions $\tilde{f}_i$ together for $i=1, \dots, n$, we obtain the desired
conformal homeomorphism $\tilde{f}: S \to S'$ extending $f$. This completes the
proof of uniqueness.

\begin{figure}[t]
	\begin{center}
		\includegraphics[scale=0.5]{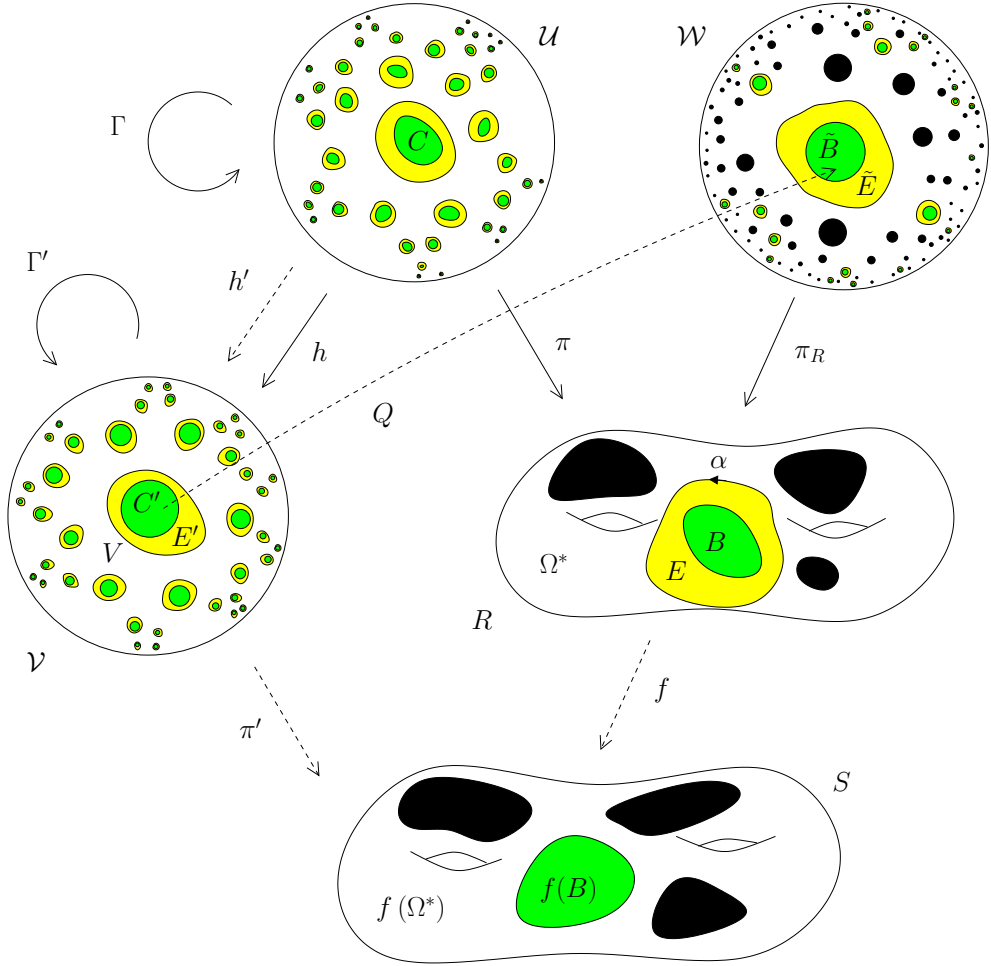}
	\end{center}
\caption{Domains, maps and group actions in the existence proof, illustrating the case where $R$ is hyperbolic.}
	\label{F: existence proof}
\end{figure}

\section{Existence}\label{S: existence}

The first step is to apply Theorem~\ref{T: HS2} to map $\Omega$ via a conformal homeomorphism onto a circle domain $\Omega^*$ in some compact Riemann surface $R$. 

Our goal is to construct a quasiconformal homeomorphism $f$ of $R$ onto
another Riemann surface $S$ so that $f$ is conformal on $\Omega^*$ and the
image $f(\Omega^*)$ is an intrinsic circle domain in $S$. In fact we will
construct a Beltrami coefficient $\mu$ on $R$ with $\|\mu\|_{\infty} < 1$;
then by the measurable Riemann mapping theorem we will obtain a Riemann surface
$S$ and a quasiconformal homeomorphism $f: R \to S$ such that $\mu(z) f_{z}=
f_{\overline{z}}$ a.e. on $R$. By construction, $\mu$ will be identically zero on
$\Omega^*$, so that $f$ is conformal there.

 Consider any connected component $B$ of $R \setminus \Omega^*$ that is not a
 single point. Our aim is to construct $\mu$ on $R$ so that $f(B)$ will be a
 closed geometric disc with respect to the Riemann surface $f(\Omega^* \cup
 B)$. Since this is a condition on the conformal structure of $f(\Omega^* \cup
 B)$, it depends only on the restriction of $\mu$ to $\Omega^* \cup B$, which
 by construction will be non-zero only on $B$. This is a key point, for it
 means that we can correct the conformal structure on each complementary
 component separately and the corrections will not interfere with each
 other. This is in contrast to the proof of Koebe's theorem by iterated
 Riemann mapping, where each complementary component has to be corrected
 infinitely many times and the required mapping is obtained in the limit. By
 hypothesis $\Omega^*$ has only finitely many complementary components in $R$,
 so to ensure that $\|\mu\|_{\infty} < 1$ it will suffice that $\|\mu\|_{B} <
 1$ for each complementary component $B$. 

 Let $\pi: \mathcal{U} \to \Omega^* \cup B$ be an unbranched analytic covering
 map, where $\mathcal{U}$ is one of $\mathbb{D}$, $\mathbb{C}$, or $\Chat$. We rule out the case $\mathcal{U} = \Chat$, since in this case
 $\Omega^*$ is an open disc in $\Chat$, so it is already an intrinsic circle
 domain. Let $\Gamma$ be the deck transformation group of $\pi$, i.e. the
 (infinite) group of conformal automorphisms $\gamma: \mathcal{U} \to
 \mathcal{U}$ such that $\pi \circ \gamma = \pi$. The preimage $\pi^{-1}(B)$
 has infinitely many connected components, each of which is a topological
 closed disc, bounded by an analytic Jordan curve. 

Apply Theorem~\ref{T: HS1} to the domain $\pi^{-1}\left(\Omega^*\right) = \mathcal{U}
\setminus \pi^{-1}(B)$. This provides a conformal mapping $h:
\pi^{-1}(\Omega^*) \to V \subset \Chat$, where $V$ is a circle domain. Each component of $\pi^{-1}(B)$ corresponds under $h$ to a component of
$\Chat \setminus V$ that is a closed disc, not a singleton, since it is isolated and cannot be
separated from the other components by a ring domain of arbitrarily large
modulus contained in $\pi^{-1}\left(\Omega^*\right)$. The remaining end of $\pi^{-1}\left(\Omega^*\right)$ may correspond to either a
point or a disc in the complement of $V$; this component is distinguished
because it is not isolated. Let $\mathcal{V}$ be the union of $V$ and all of
its isolated complementary components.

For any non-trivial element $\gamma \in \Gamma$, the map $h \circ \gamma$ is
also a conformal mapping of $\pi^{-1}\left(\Omega^*\right)$ onto the circle domain $V$, so by the uniqueness part
of Theorem~\ref{T: HS1} we have $h \circ \gamma = \gamma' \circ h$ for a
unique M\"{o}bius map $\gamma'$, which restricts to an automorphism of $V$
with no fixed points since $\gamma$ has no fixed points in
$\mathcal{U}$. The map that sends $\gamma$ to $\gamma'$ is therefore an
injective homomorphism $\Gamma \to \Gamma'$, where $\Gamma'$ is a subgroup of
$\textup{Aut}(\mathcal{V})$. Since $\Gamma'$ acts simply transitively on the
components of $h(\pi^{-1}(B))$ and has no fixed points in $V$,  it acts freely on $\mathcal{V}$. Since $\textup{Aut}(V)$ acts properly discontinuously on $V$, and each compact subset of $\mathcal{V}$ intersects only finitely many components of $h(\pi^{-1}(B))$, we also conclude that $\Gamma'$ acts properly discontinuously on $\mathcal{V}$.

 Let $C$ be a  connected component of $\mathcal{U} \setminus
 \pi^{-1}\left(\Omega^*\right)$, and let $C'$ be the corresponding component of
 $\mathcal{V} \setminus V$, so that $C'$ is bounded by a circle $\partial
 C'$.   Let $\pi_R: \mathcal{W} \to R$ be an analytic universal covering of
 $R$, so that $\pi_R^{-1}\left(\Omega^*\right)$ is a circle domain. Let
 $\tilde{B}$ be a connected component of $\pi_R^{-1}(B)$, bounded by a circle
 $\partial \tilde{B}$.

 We will examine the behavior of the maps $h, \pi$ and $\pi_R$ in the
 neighborhood of the curves $\partial C$ and $\partial \tilde{B}$. For the
 present proof the aim is merely to show that $h$ may be extended continuously to a quasiconformal homeomorphism $h': \mathcal{U} \to \mathcal{V}$ that is
 equivariant with respect to $\Gamma$ and $\Gamma'$. In fact we will work
 slightly harder, using Lemma~\ref{L: distortion bound} to show that the quasiconformal dilatation of this extension
 can be made to depend only on a conformal invariant of $\Omega$. This will be
 useful in the next section.

 Let $\alpha$ be the simple closed geodesic in the hyperbolic metric of $\Omega^*$ that separates
 $B$ from all the other components of $R \setminus \Omega^*$. Let $E$ be the
 ring domain bounded between $\alpha$ and $B$. It has modulus $\textup{mod}(E)
 = \pi^2/2\ell(\alpha)$, where $\ell(\alpha)$ is the length of $\alpha$ in the
 hyperbolic metric of $\Omega^*$; note that this length is a conformal
 invariant of $\Omega$.

 The restriction of $\pi_R$
to this circle domain is an unbranched analytic covering of $\Omega^*$ and is therefore
an isometry from the hyperbolic metric of $\pi_R^{-1}\left(\Omega^*\right)$ to
the hyperbolic metric of $\Omega^*$. Thus the connected components of $\pi_R^{-1}(\alpha)$ are simple closed geodesics in
 the hyperbolic metric of $\pi_R^{-1}(\Omega^*)$ and are therefore disjoint,
 each separating a connected component of $\pi_R^{-1}(B)$ from all of the
 other components of $\mathcal{W} \setminus
 \pi_R^{-1}\left(\Omega^*\right)$. It follows that $\pi_R$ maps each connected
 component of $\pi_R^{-1}(E)$ bijectively onto $E$. One of these is a ring
 domain $\tilde{E}$ surrounding $\tilde{B}$.

Likewise, the restriction of $\pi \circ h^{-1}$ to $V$ is an unbranched
analytic covering of $\Omega^*$, so $\pi \circ h^{-1}$ maps each connected
component of $h \circ \pi^{-1}(E)$ bijectively onto $E$. One of these connected components is a ring domain $E'$ surrounding $C'$.

 Consider the branch of $\pi_R^{-1} \circ \pi \circ h^{-1}$ that maps $E'$ to $\tilde{E}$.  It is a homeomorphism of ring domains taking the inner boundary
 circle $\partial C'$ onto the inner boundary circle $\partial \tilde{B}$.
 Lemma~\ref{L: distortion bound} gives an extension to a quasiconformal
 homeomorphism $Q$ from $E' \cup C'$ to $\tilde{E} \cup \tilde{B}$. 

Now $Q^{-1} \circ \pi_R^{-1}$ gives us a quasiconformal homeomorphism  $g: B \to C'$
that continuously extends the boundary correspondence induced by of $h \circ \pi^{-1}$. We define $\mu|_B$ to be the Beltrami coefficient $g_{\overline{z}}/g_{z}$.

When we solve the Beltrami equation to obtain $f: R \to S$, the covering map
$\pi': \mathcal{V} \to \left(S \setminus f\left(\Omega^*\right) \right) \cup
f(B)$ such that $f \circ \pi = \pi' \circ h$ will be an unbranched analytic
covering, mapping the disc $C'$ onto $f(B)$, as required.

\section{The case of  countably infinite connectivity}\label{S: countable}

It is natural to ask whether the notion of intrinsic circle domains can be
extended to domains of countably infinite connectivity.

\subsection{All ends isolated}

First, we deal with the case of a domain $\Omega$ such that every end of $R$ is
\emph{isolated}, i.e. has a neighborhood that meets no other end. If $\Omega$
is embedded in a compact ambient Riemann surface then this implies that
$\Omega$ has only finitely many ends and finite genus, and we already
understand this case. 

For a more interesting example, let $R$ be a $\mathbb{Z}$-cover of a compact
genus $2$ Riemann surface, so that $R$ has infinite genus and two ends. Then
let $\Omega$ be a domain obtained by removing countably many disjoint
closed topological discs with no accumulation point from $R$. We could hope to
modify the structure of $R$ on a neighborhood of each disc in order to make
the complementary components be intrinsic discs. However, there is no way to fix
the two ends of $R$ so that they are represented by intrinsic discs, for the
resulting ambient Riemann surface would be compact and therefore have only
finite genus.

Suppose that $\Omega$ is infinitely connected or has infinite genus.
We call an end of $\Omega$ \emph{fixable} if it has a neighborhood that is a
ring domain.

\begin{lemma}
 For any Riemann surface $\Omega$, there is a conformal embedding of $\Omega$
 into another Riemann surface $R$ such that every connected component of
 $R \setminus \Omega$ is an intrinsic disc and no end of $R$ is fixable.
\end{lemma}
\begin{proof}
 We can find disjoint open neighborhoods of all the fixable ends,
for example by cutting along closed geodesics or horocycles around each end. 

  In the existence proof for Theorem~\ref{T: general}
we only used a local surgery to modify the ambient
Riemann surface in a neighborhood of each complementary component. The
initial step of passing to a circle domain was technically convenient (and
will be needed later), but was not really necessary. All we really needed to know was that for each individual
end $E$, we can embed $\Omega$ in a Riemann surface $R_E$ such that $R_E
\setminus \Omega$ is a geometric disc in $R_E$. This is true for any fixable
end $E$, since it has a neighborhood that is a ring domain and therefore it
has a neighborhood $U_E$
conformally equivalent to a round annulus $A = \{ z \in \mathbb{C}\,|\, a <
|z| < 1\}$, for some $0 \le a < 1$. Then we can glue the open unit disc to
$\Omega$, identifying $A$ with $U_E$, to obtain a Riemann surface $R_E'$ in
which $\Omega$ is embedded so that the end $E$ corresponds to a connected
component $K$ of $R_E' \setminus \Omega$. Then we can apply Theorem~\ref{T: HS1}
to the preimage of $\Omega$ in the universal cover of $R_E'$, to find out how
to modify the conformal structure on $K$ to obtain an
embedding of $\Omega$ in a new Riemann surface $R_E$ so that the end $E$
corresponds to a connected component of $R_E \setminus \Omega$ that is an
intrinsic disc. 

 To construct $R$ we glue together all the Riemann surfaces $R_E$
 corresponding to fixable ends $E$ by identifying the embedded copies of
 $\Omega$. The resulting $R$ is locally a Riemann surface, is connected, and is second countable because $\Omega$ can only have countably many fixable ends. 
\end{proof}

\subsection{Non-isolated ends}

 We now return to the case of a subdomain $\Omega$ of a compact Riemann
 surface, but assume that $\Omega$ is countably infinitely connected. It must now be the
 case that some complementary components are not isolated.  There is a potential topological obstruction associated with complementary
 components that are not isolated: they cannot be represented by intrinsic
 discs of finite radius. We might attempt to salvage something by allowing complementary
components to be horodiscs with respect to the hyperbolic metric, but this
would not help in the case of a circle domain $\Omega \subset
\Chat$ in which some circular complementary component $B$ has precisely two
points on its boundary that are accumulation points of other complementary
circles. 

Therefore in the countably infinitely connected case, we define an intrinsic
circle domain $\Omega$ to be a subdomain of a compact Riemann surface $R$ such
that any complementary component that is not isolated is a singleton and
every non-singleton complementary component $L$ is a closed geometric disc
with respect to $\Omega \cup L$. This definition is intended to be analogous
to the definition of a circle domain in the countably connected case.

In order to obtain a positive theorem, we can place a simple
conformal geometric constraint on the domain.  We will say that a countably
connected domain $\Omega$ in a compact Riemann surface $R$ is \emph{uniformly
  separated} if there exists $\epsilon > 0$ such that each connected component of $R \setminus \Omega$ either is a single point or is separated from
  all the remaining complementary components by a ring domain of modulus at
  least $\epsilon$ embedded in $R \setminus \Omega$.

\begin{lemma}\label{L: no cavitation}
 Suppose $\Omega$ is a countably-connected uniformly separated domain in a
 compact Riemann surface $R$. Let $\varphi$ be any conformal embedding of
 $\Omega$ into a compact Riemann surface $R'$ of the same genus as $R$. Then the
 non-singleton complementary components of $R'$ correspond to the
 non-singleton complementary components of $R$, so  $\varphi(\Omega)$
 is uniformly separated in $R'$.
\end{lemma}

To explain the significance of this, we first note that a complementary component $B$
is isolated if and only if $\Omega$ contains a ring domain with $B$ as one of
its complementary components; this is a topological condition on $\Omega$, so
the corresponding complementary component in $\varphi(\Omega)$ is also not
isolated. If $\Omega$ is uniformly separated, then the connected components of
$R \setminus \Omega$ that are not singletons must be isolated. However, the
remaining complementary components need not be punctures; some or all of them
could be accumulation points of sequences of other complementary components.  For a general open Riemann surface $\Omega$ of countable connectivity and
finite genus, it is possible for an end of $\Omega$ to be represented in
one conformal embedding by a complementary component that is a non-isolated
singleton, yet in some other conformal embedding to be represented by a
non-singleton. 

\begin{figure}[h]
	\begin{center}
		\includegraphics[scale=0.5]{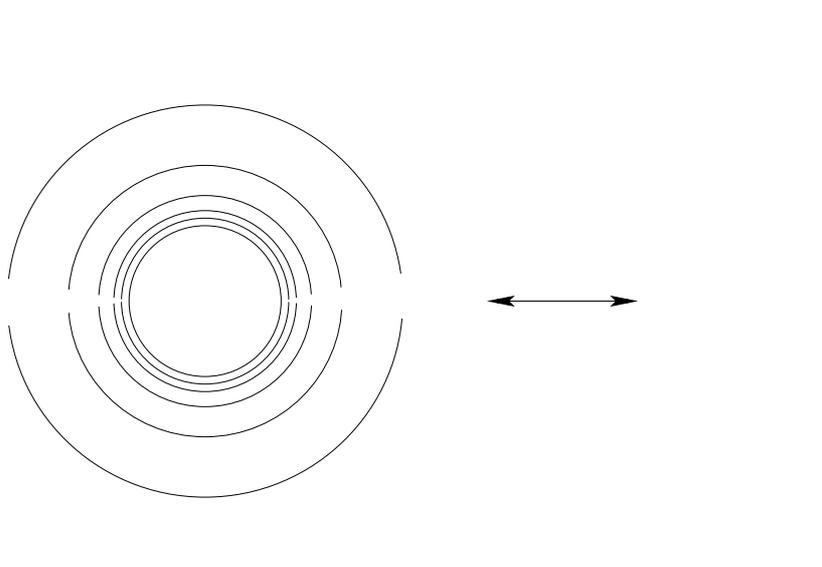}
	\end{center}
\caption{Two conformally equivalent countably connected domains.}
	\label{F: cavitation}
\end{figure}

 We can construct an example of this
\emph{cavitation} behavior\footnote{The term cavitation refers in fluid dynamics to the sudden formation of a bubble as
  a dissolved gas comes out of solution around a nucleation site.} as
follows.  Consider a subdomain $D$ of $\mathbb{C} \setminus \overline{\mathbb{D}}$ obtained
by removing countably many complex conjugate pairs of circular arcs,
where one arc of the $n^{\textup{th}}$ pair is 
\[ A_n \,=\,\left\{\left(1+2^{-n}\right)e^{it}\,: \theta_n < t < \pi -
      \theta_n\,\right\}\,, \quad \theta_n \searrow 0\,. \]  The domain
    \[D_n\, =\, \mathbb{C} \setminus \left( \mathbb{D} \cup \bigcup_{j=1}^n \left(A_j \cup \overline{A_j}\right) \right)\]
is conformally equivalent via a conformal map that fixes $z=2$ to a slit
domain $S_n$, which may be obtained from
$\mathbb{C}\setminus\{0\}$ by removing finitely many pairs of intervals $I_n$,
$\overline{I_n}$ contained in
the imaginary axis, arranged symmetrically about $0$, together with an
interval $(-\delta_n i, \delta_n i)$. See figure~\ref{F: cavitation} for a
schematic illustration. One
can show that if the angles $\theta_n$ are chosen to decrease to $0$
sufficiently fast then the domains $S_n$ converge in the Carath\'{e}odory
topology to a limit domain $S$ in which $0$ is a non-isolated singleton
complementary component, but every other complementary component is an
isolated interval of the imaginary axis.  Then $D$ is conformally equivalent
to $S$ and we have an example of cavitation: the singleton complementary
component $\{0\}$ of $S$ corresponds to the complementary component
$\overline{\mathbb{D}}$ of $D$. 

The idea is that by making $\theta_n$ small we can ensure that the extremal length of the family of
curves in $D$ joining $A_n$ to $\overline{A_n}$ is as small as we like. This
means we can make the gap between the intervals $I_n$ and $\overline{I_n}$ as
short as we like in comparison to the length of $I_n$. In fact to ensure that
the end represented in $D$ by the unit circle is represented in $S$ by a
singleton, it suffices to take $\theta_n = 2^{-n}$.

We will now prove Lemma~\ref{L: no cavitation}, which says that uniform separation prevents cavitation.

\begin{proof}
Suppose each non-singleton connected component $B_i$ of $R \setminus
\Omega$ is separated from the remaining complementary components by a ring
domain $A_i \subseteq \Omega$ of modulus $\epsilon$. Then the core
curve $\gamma_i$ of the ring domain $A_i$ has length $2\pi^2/\epsilon$ with
respect to the hyperbolic metric of $A_i$. By the Schwarz-Pick lemma, the
length of $\gamma_i$ with respect to the hyperbolic metric of $\Omega$ is also
at most $2\pi^2/\epsilon$. Consider the simple closed geodesic $g_i$ (with
respect to the hyperbolic metric of $\Omega$) that lies in the
free homotopy class of $\gamma_i$. Since $g_i$ is a length minimizer in its
free homotopy class, it is no longer than $\gamma_i$ in the hyperbolic metric
of $\Omega$. It follows that the ring domain $E_i$ bounded between $g_i$ and
$B_i$ has modulus at least $\epsilon/4$. Our reason for passing from the ring
domains $A_i$ to the ring domains $E_i$ is that the closures of
the $E_i$ in $R$ are pairwise disjoint, which was not necessarily true of the
$A_i$, and moreover the $E_i$ are intrinsically-defined subdomains of $\Omega$. 

 We now cut $\Omega$ along each $g_i$ to leave a domain $U
= \Omega \setminus \bigcup \overline{E_i}$.  We claim that for any conformal
embedding $\varphi$ of $\Omega$ into a compact Riemann
surface $R'$ of the same genus as $R$, there exists a quasiconformal
homeomorphism $Q: R \to R'$ such that $\varphi|_{U} = Q|_{U}$. To prove this,
consider a component $D_i = B_i \cup E_i$ of $R \setminus U$. It corresponds
to a component $D_i' = B_i' \cup \varphi(E_i)$ of $R' \setminus \varphi(U)$,
where $B_i'$ is a component of $R' \setminus \varphi(\Omega)$. We claim
that there is a $K$-quasiconformal homeomorphism $D_i \to D_i'$ whose
restriction to the curve $g_i$ agrees with $\varphi$. Here $K$ depends only
on $\epsilon$. To see this, let $\psi: D_i \to \mathbb{D}$ and $\chi: D_i' \to
\mathbb{D}$ be Riemann mappings, and note that $\partial D_i$ and $\partial
D_i'$ are analytic Jordan curves. Apply Lemma~\ref{L: distortion bound} to the conformal
homeomorphism \[\chi \circ \varphi \circ \psi^{-1}\,: \,\psi\left(E_i\right)
\to \chi\left(E_i'\right)\] and apply a Schwarz reflection to obtain a
$K$-quasiconformal homeomorphism $h: \mathbb{D} \to \mathbb{D}$ whose boundary
correspondence agrees with that of $\chi \circ \varphi \circ \psi^{-1}$, where
$K$ depends only on $\epsilon$. Now
\[\chi^{-1} \circ h \circ \psi\,:\, D_i \to D_i'\] is a $K$-quasiconformal
homeomorphism agreeing with $\varphi$ on the boundary, as required.

Gluing together $\varphi$ and the $K$-quasiconformal extensions to the
components $D_i$, we
obtain a $K$-quasiconformal homeomorphism \[ H: \Omega \cup \bigcup B_i \;\to\;
\varphi(\Omega) \cup \bigcup B_i'\,.\] 

Since $\Omega$ is countably connected, the complement $X = R \setminus
\left(\Omega \cup \bigcup B_i\right)$ is a countable union of singletons, and
therefore is a removable set for the solution of the Beltrami
equation. Thus $H$ extends to a homeomorphism $R \to R'$, and hence the
 complementary components of $R'$ that are not among the $B_i$ are singletons.

\end{proof}

\begin{theorem}\label{T: infinite} Suppose $\Omega$ is a domain in a compact
  Riemann surface $R$ such that $\Omega$ has countably infinite
  connectivity, has the same genus as $R$, and is
  uniformly separated. Then there is a conformal embedding $\varphi$ of $\Omega$ in a compact Riemann surface $S$ so that $\varphi(\Omega)$ is an intrinsic circle  domain in $S$. The pair $(S, \varphi)$ is unique up to conformal isomorphism.
\end{theorem}

\begin{proof}
 The existence proof is similar to
that of Theorem~\ref{T: general}. Lemma~\ref{L: no cavitation} shows that when
we apply Theorem~\ref{T: HS2}, the resulting countably connected circle domain
$\Omega^* \subset R$
satisfies the hypotheses of Theorem~\ref{T: infinite}, in addition to having
the property that its non-singleton complementary components are geometric
discs. Then we note that by Lemma~\ref{L: distortion bound} the quasiconformal distortion of the extension $C \to B$ can be bounded in terms of $\epsilon$. Then the
  resulting Beltrami coefficient $\mu$ has $\|\mu\|_\infty < 1$, so
  we may still apply the measurable Riemann mapping theorem to obtain
  the desired conformal structure on $S$. Then Lemma~\ref{L: no cavitation}
  ensures that the non-isolated complementary components of $f(\Omega^*)$ are
  singletons, so that $f(\Omega^*)$ is an intrinsic circle domain. 

  The uniqueness proof
  goes through as before, but with one additional step. Because of the
  separation hypothesis there is no difficulty in extending a conformal
  homeomorphism $f: \Omega \to \Omega'$ between two countably-connected
  intrinsic circle domains across each non-singleton complementary
  component. We thus obtain a conformal homeomorphism between domains
  whose complements consist of countably many points, which in turn extends
  to a conformal homeomorphism 
  $\tilde{f}: S \to S'$ between the ambient Riemann surfaces. 
\end{proof}

It is likely that the condition of uniform separation can be weakened by
making use of results on existence and uniqueness of solutions to the Beltrami
equation in the case where the Beltrami coefficient has norm $1$ but the
measure of the set on which the distortion is large is controlled.

\section{Mixtures of intrinsic and extrinsic circles}\label{S: mixed}

Now we present a generalization of Theorems~\ref{T: general}~and~\ref{T: infinite} in which the conditions on the complementary components of
$\Omega$ in $S$ are mixed. That is, some components are punctures or geometric
discs in the natural geometry of the ambient compact Riemann surface $S$,
while each of the others is an intrinsic disc, i.e. a geometric disc with respect to its own union with $\Omega$.

\begin{theorem}\label{T: mixed}
Let $\Omega$ be an open Riemann surface with finite genus and at most countably
many ends. Let $K_i$, $i \in I$ be some of the ends of $\Omega$, none of them
punctures, where the index set $I$ may be either finite or countably
infinite. Suppose that for some $\epsilon > 0$ and for each $i \in I$ there is
a ring domain $A_i$ contained in $\Omega$, with modulus at least $\epsilon$, that separates
the end $K_i$ from the other ends of $\Omega$. Then $\Omega$ is
conformally equivalent to a domain $\varphi(\Omega)$ in a compact Riemann surface $S$,
with complementary components $L_i$ corresponding to $K_i$, such
that $\varphi(\Omega) \cup \bigcup_{i \in I} L_i$ is a circle domain in $S$
and for each $i \in I$, $L_i$ is either a singleton or a closed geometric disc
with respect to the domain $\varphi(\Omega) \cup L_i$. Such a pair $(S,\varphi)$ is unique up to conformal isomorphism.
\end{theorem}

\begin{proof}
 We begin by replacing $\Omega$ by a conformally equivalent circle domain
 $\Omega^*$ in a compact Riemann surface $R$, as in Theorem~\ref{T: HS2}.  Apply the construction of section~\ref{S: existence} to each of the components $K_i$, $i \in I$, to obtain a Beltrami differential $\mu(z) \frac{dz}{d\overline{z}}$ supported on $\bigcup K_i$. In the case where $I$ is infinite, the separation condition guarantees that $\|\mu\|_\infty < 1$, as it did for Theorem~\ref{T: infinite}. Solving the Beltrami equation $\mu(z) f_z  = f_{\overline{z}}$ on $R$ gives a new compact Riemann surface $S'$ with a homeomorphism $f: R \to R'$, conformal away from the preimages of the $K_i$, such that $f(K_i)$ is a closed hyperbolic disc with respect to the hyperbolic metric of $f(U) \cup f(K_i)$. The domain $f\left(U \cup \bigcup K_i\right)$ is countably connected so by Theorem~\ref{T: HS2} there is a conformal homeomorphism $g: f\left(U \cup \bigcup K_i\right) \to \Omega^*$, where $\Omega^*$ is a circle domain in a compact Riemann surface $S$. The image $\Omega = g(f(U))$ is the required domain, and $L_i = g(f(K_i))$. Indeed, $g$ is conformal on $U \cup f(K_i)$, so it is an isometry from the hyperbolic metric of $U \cup f(K_i)$ to the hyperbolic metric of $\Omega \cup L_i$.

 To prove uniqueness, suppose that $f: \Omega \to \Omega'$ is a conformal
 homeomorphism between two domains $\Omega \subset S$ and $\Omega' \subset S'$, each satisfying the conditions of the
 theorem. Let the complementary components of $\Omega'$ be $L_i'$, $i \in I$. Apply the argument of \S\ref{S: uniqueness} to extend $f$ to a
 conformal homeomorphism of countably connected circle domains $f: \Omega \cup \bigcup L_i \to \Omega' \cup \bigcup L_i'$. This must be the restriction of a conformal map from $S$ to $S'$ by the final part of He and Schramm's Theorem~\ref{T: HS2}.
\end{proof}

\section{Numerical approximation using circle packings}\label{S: pictures}

The rest of this paper concerns the practical numerical approximation of
finitely-connected intrinsic circle domains in the sphere. The idea is roughly to follow the steps
of the existence proof, performing each step numerically.

Firstly, we need a way to approximate circle domains having Fuchsian
symmetry. This already presents us with a choice of methods, as there are
several numerical methods for computing circle domains. Finitely connected
circle domains can be approximated using iterated Riemann mapping, the
individual steps of which can be carried out using a number of different
methods for computational conformal mapping. However, we must deal with an
infinitely connected circle domain, so it seems at first sight that we must
make decisions about which complementary components to ignore at any
particular stage of the approximation. We will see that this truncation can in
fact be avoided.

Secondly, if we follow the existence proof directly, it seems that we need to
compute a quasiconformal extension numerically, find its dilatation and then
solve a Beltrami equation on a larger domain. In practice, these steps can be
combined into a single \emph{conformal welding} step. 

 A final practical difficulty may arise in converting the output of each step
 into the input for the next step, especially when the steps require different
 types of grid.

Fortunately, circle packing is a numerical approximation method that offers
solutions to all of these challenges. Although circle packing is not
often used for high-precision conformal mapping problems on account of its
relatively slow convergence, it is particularly well suited to the
approximation of circle domains, especially those with Fuchsian symmetry. It
is also a practical tool for numerical conformal welding. We were able to
carry out all of the steps of an approximation procedure for computing
intrinsic circle domains using the \texttt{CirclePack} software \cite{CP}
written by Ken Stephenson et al. For an introduction to circle packing and its use as a computational tool, we refer the reader to Stephenson's monograph~\cite{St}.   

\subsection{Overview of circle packing}

Let $T$ be a finite graph embedded in the plane. A circle packing of $T$ in the plane is
a collection $\mathcal{P}$ of circular discs $C_v$ with disjoint interiors, one corresponding to each
vertex $v$ of $T$, such that whenever vertices $v$ and $w$ are adjacent in $T$
the discs $C_v$ and $C_w$ are tangent. We call $T$ the \emph{nerve} of
$\mathcal{P}$.  We may also circle pack in the sphere, using discs in the
spherical metric, or in the hyperbolic plane, using hyperbolic circles and
possibly also horocircles on the boundary. The \emph{carrier} of $\mathcal{P}$ is the geometric complex
formed by connecting the centers of neighboring circles by geodesic
segments. According to the Koebe-Andreev-Thurston theorem, when $T$ is any
triangulation of the sphere, there exists a circle packing of $T$, and it is
unique up to M\"{o}bius maps and reflection. In this case, the carrier of
$\mathcal{P}$ is the entire sphere. Since we want to use circle packings to
approximate conformal structures, we remove the reflection ambiguity by
imposing a fixed orientation. 

There is a beautiful algorithm due to Bill Thurston for computing a
circle packing from a given nerve. It works by removing one vertex and packing
the remaining triangulation into the unit disc. This is achieved by
computing the \emph{hyperbolic packing label} for the packing; this is the function
which assigns a hyperbolic radius to each vertex of the triangulation, in such
a way that the boundary circles are given infinite radius, so that they
correspond to horocircles, internally tangent to the unit circle, and all
interior circles have finite hyperbolic radius. There is a unique packing
label that results in an angle sum of $2\pi$ at each interior
vertex. Thurston's algorithm approximates the correct label as the limit of a
pointwise increasing sequence of labels. Once the correct packing label is
known, the circles of the given radii can be laid out in the hyperbolic plane
iteratively so that each satisfies the appropriate tangency conditions. 

\subsection{Outline of the approximation algorithm}
Here we outline a numerical method for approximating
intrinsic circle domains. Let $\Omega$ be a given finitely connected domain in
the Riemann sphere. Begin by taking a circle packing $\mathcal{P}$ such that the
interiors of all the circles of $\mathcal{P}$ are contained in $\Omega$ and the
complementary components of $\Omega$ are separated by the carrier of
$\mathcal{P}$. All subsequent calculations depend only on the nerve of
$\mathcal{P}$, so it is in this step that we have captured a discrete
approximation to the conformal equivalence class of $\Omega$. The quality of the final approximation of the intrinsic circle domain $\Omega^*$ conformally equivalent to $\Omega$ will depend on the mesh of $\mathcal{P}$ (the size of its largest circle) and the maximum distance of $\partial \Omega$ from the carrier of $\mathcal{P}$.

  We perform a sequence of circle packing computations in order to construct a
  triangulation of the sphere together with an embedding of the nerve of
  $\mathcal{P}$ as a subcomplex. When we compute the circle packing
  $\mathcal{Q}$ of the spherical triangulation, the carrier of the embedded
  sub-packing gives an  approximation to the intrinsic circle domain $\Omega$
  conformally equivalent  to $U$. We can interpolate the mapping of circle centers of $\mathcal{P}$ to the corresponding circle centers in $\mathcal{Q}$
  to give a \emph{polyhedral} embedding of the carrier of $\mathcal{P}$ into
  the carrier of $\mathcal{Q}$. Since we are mapping from the plane to the
  sphere, we cannot map in a piecewise affine fashion, so instead we map in a
  piecewise affine fashion to the polyhedron in $\mathbb{R}^3$ whose vertices
  are the vertices of $\mathcal{Q}$ in $S^2$, and then project radially
  outwards to $S^2$.  This map gives a homeomorphism which approximates the
  conformal map from $U$ to $\Omega$ and is locally quasiconformal.

 We will explain informally why each step of the computation provides an
 arbitrarily good  approximation to a conformal map occurring at the corresponding step in the
 constructive proof of Theorem~\ref{T: main}. Note that the computation
 consists of finitely many circle packing steps. In each step we are
 approximating the solution of a problem whose solution depends continuously
 on its data. This should suffice to prove local uniform convergence of the
 approximations to the conformal map from $\Omega$ to $\Omega^*$, but we will not attempt to estimate the rate of convergence.

\subsection{Approximating circle domains via circle packing}

Suppose we are given a bounded finitely connected domain $\Omega$ in the complex
plane. We find a sequence of circle packings $\mathcal{P}_n$ in $\Omega$ whose
carriers exhaust $\Omega$. To construct $\mathcal{P}_n$ we \emph{cut out} a
portion of the regular hexagonal circle packing with circles all of radius
$\epsilon = 2^{-n}$. This means that we keep only those circles whose
interiors are entirely contained in $\Omega$. Stephenson \cite{St} describes this as using
$\partial\Omega$ as a `cookie cutter'. We have to take some care at the
boundary: we retain only the largest connected component of the hexagonal
packing that remains, discarding any peripheral islands and iteratively
removing vertices of degree one until we are left with a connected
triangulation. When we take $n$ large enough, the boundary components
of this triangulation will be in one-to-one correspondence with the complementary components of $\Omega$. 

Next, to each cycle of boundary vertices of $\mathcal{P}_n$ we adjoin a new `ideal' vertex, adjacent to all
of the vertices in the cycle, and in this way we obtain a triangulation of the
sphere. By the Koebe-Andreev-Thurston theorem there is a circle packing whose
nerve is this triangulation, and it is unique up to M\"{o}bius
transformations. Consider what happens to this spherical packing as we let
$n$ tend to $\infty$. Using quasiconformal distortion estimates
for packings of bounded vertex degree, one can show
that after a suitable M\"{o}bius normalization, the circles corresponding to
the added ideal vertices converge to limit circles of positive radius, while
the maximum radius of the all the other circles tends to zero uniformly.

 Let $f_n$ be the polyhedral map from the hexagonal
cut-out packing $\mathcal{P}_n$ to the carrier of the spherical packing, normalized so that
three chosen points in $\Omega$ are fixed by $f_n$.  It turns out that as $n \to
\infty$, $f_n$ converges locally uniformly on $\Omega$ to a conformal map
from $\Omega$ to a circle domain. To prove this, we can use the hex packing
lemma of Rodin and Sullivan \cite{RS}. This says that the quasiconformal
dilatation of $f_n$ converges locally uniformly to $0$ on
$\Omega$. The Rodin-Sullivan bound for the quasiconformal dilation on a
given face of the triangulation depends only on the number of layers of
hexagonal packing surrounding it, and tends to $0$ as the number of layers
tends to infinity.  Let $\Omega'$ be any subdomain of $\Omega$ bounded by Jordan
curves in $\Omega$, such that $\Omega'$ is homeomorphic to
$\Omega$. The dilatation bound shows that the images of the maps $f_n$
restricted to $\Omega'$ remain
within a compact subset of the Teichm\"{u}ller space of multiply connected
domains homeomorphic to $\Omega$. Any subsequential limit of
the sequence $f_n$ must be a conformal homeomorphism, and its image
must be a circle domain. Given the normalization, there is a unique such
homeomorphism. It follows that the sequence $f_n$ converges locally uniformly
as $n \to \infty$. 

We can augment the image packings by inverting across all the `ideal'
circles, adding circles of degree $4$ in the resulting four-sided
interstices to maintain a triangulation. By considering moduli of ring domains
in the resulting packing it is possible to show that the maximum radius of any
non-ideal circle in the image packing also converges to $0$ as $n \to \infty$.  

We have now seen how to compute an arbitrarily good approximation to the
circle domain $\Omega^*$ conformally equivalent to $\Omega$. We may assume
that $\Omega^*$ is contained in the plane. 

\begin{figure}[!h]
	\begin{center}
		\includegraphics[scale=0.5]{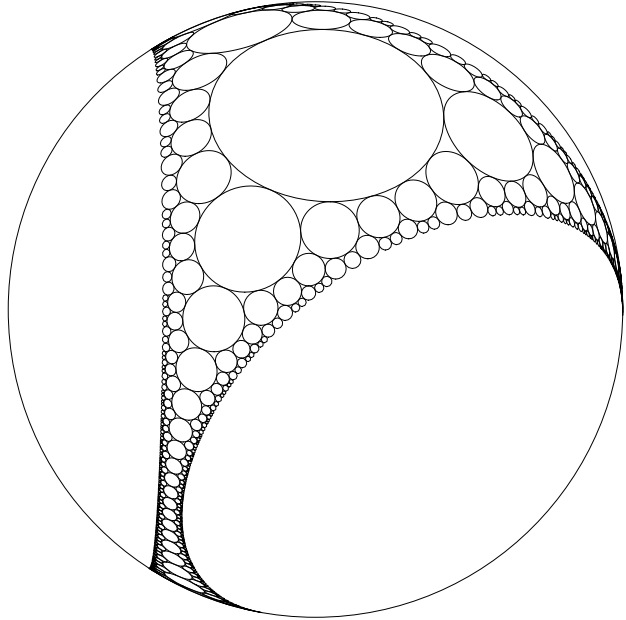}
	\end{center}
\caption{A circle packing of a circle domain in the Riemann sphere. Both the
  domain and the circle packing have sixfold symmetry.}
	\label{F: circle domain}
\end{figure}

\subsection{Circle domains with Fuchsian symmetry}

 Suppose $B$ is one of the
complementary components of $\Omega^*$. Let $\pi: \mathbb{D} \to \Omega^* \cup
B$ be an analytic universal covering map. To compute an approximation to the
circle domain conformally equivalent to $\mathbb{D} \setminus \pi^{-1}(B)$, we use the same trick again. We have a triangulation of $\Omega^*$, to which we add just one ideal vertex $v_B$ corresponding to the complementary component $B$. Now we apply Thurston's algorithm to compute the \emph{hyperbolic} packing label for the resulting multiply connected triangulation. If we lay the circles out according to this packing label, we get a Fuchsian monodromy group. Call the resulting packing $\mathcal{P}_B$. The circle corresponding to $v_B$ and its translates under the monodromy group approximate the bounded components of the required circle domain. For the conformal welding step that follows we will only need to know the hyperbolic radii for the circles corresponding to the vertex $v_B$ and its neighbors. It is not necessary to perform the layout routine.

\begin{figure}[!h]
	\begin{center}
		\includegraphics[scale=0.5]{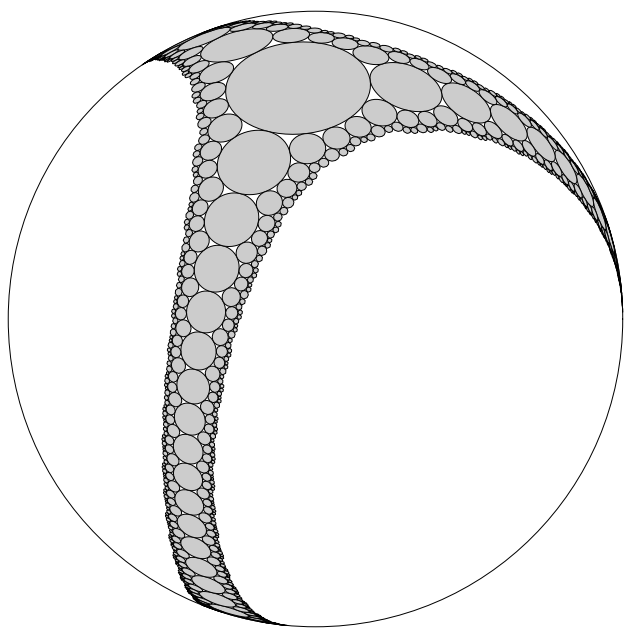}
	\end{center}
\caption{A circle packing approximation of the intrinsic circle domain
  conformally equivalent to the circle domain in Fig.~\ref{F: circle domain},
  with the circle packing of the same complex shown.}
	\label{F: intrinsic circle domain}
\end{figure}

\subsection{Conformal welding via circle packing}

Given a Jordan curve $\gamma$ in the Riemann sphere, we can compute the Riemann mappings of the unit disc onto the interior and the exterior of $\gamma$. The boundary values of these Riemann mappings give two different continuous maps of the unit circle onto $\gamma$, $\phi_{int}$ and $\phi_{ext}$. The shape of $\gamma$ is encapsulated in the boundary correspondence $\phi_{ext}^{-1} \circ \phi_{int}$, which is an orientation-reversing homeomorphism of the unit circle. Conformal welding is the process of recovering $\gamma$ (up to a M\"{o}bius map) from the boundary correspondence.

The \texttt{CirclePack} software includes a package for performing approximate
conformal welding, following the procedure described by Williams \cite{W}. The idea is to paste two combinatorial closed discs
together along their boundaries, according to a best possible combinatorial
approximation of the given boundary correspondence,
to produce a triangulation of the sphere. We then compute the circle
packing for this triangulation. For example, we may start with two maximal packings of the unit disc
(meaning that the boundary circles are internally tangent to the unit circle),
and use the arc length around the unit circle as a guide for pasting together
the two triangulations, introducing new vertices where necessary to keep
control of the vertex degrees.

In our application to intrinsic circle domains, we weld a combinatorial closed
disc to the boundary cycle corresponding to each complementary component $B$
of $\Omega$, using the arc length around the circle corresponding to $v_B$ in
the packing $\mathcal{P}_B$ and the trivial boundary correspondence
$e^{i\theta} \mapsto e^{-i \theta}$. A suitable choice of combinatorial closed
disc would be a large section of the regular hexagonal packing, for then the
hex packing lemma of Rodin and Sullivan can be applied to show that the
desired welding and the computed welding differ by a quasiconformal map which
has small dilatation except on an annulus of small modulus and small area
covering the cycle of edges along which the welding is performed.  After
performing a welding for each boundary component, we simply apply Thurston's
algorithm to compute a circle packing of the resulting spherical
triangulation. In the case where the ambient Riemann surface $R$ has positive
genus, we would compute a packing of the universal cover of the welded complex at this stage to get a circle packing approximation to the Riemann surface $S$.

 The finished spherical circle packing gives an approximation to the map
 $\varphi$ and to the desired conformal structure on the interior of each complementary component of $\varphi(\Omega)$, and on the interior of the image of
 $\varphi(\Omega)$. 

\subsection{Convergence}

 We sketch a proof that the approximation scheme described above does
 converge locally uniformly on the domain $\Omega^*$ to the conformal mapping
 to an intrinsic circle domain, after suitable normalization. The strategy is
 to show that every map appearing in Figure~\ref{F: existence proof} is well
 approximated by the corresponding polyhedral map between circle packings.
 Then since each step depends continuously on its input data, we will find
 that the polyhedral map from $\mathcal{P}^*$ to $S^2$ converges locally
 uniformly on $\Omega^*$ to the desired conformal mapping.

 On the subcomplex of the packing representing the intrinsic circle domain, the hex-packing distortion estimate of
 Rodin and Sullivan shows that the polyhedral mapping is quasiconformal with
 dilatation that converges locally uniformly to zero.  We
 can apply a similar estimate on the interior of each complementary
 component. 

 It remains to show that the conformal welding step is well approximated by
 the discrete conformal welding. To do this we have to control the
 conformal modulus of the image of a narrow annular neighborhood of each
 boundary component, where the circles are not deep enough in the hex packing
 for the Rodin-Sullivan estimate to give the bound we need.  To do this it suffices to bound the maximum degree of the vertices
 appearing in the final triangulation. This is because the Rodin-Sullivan ring
 lemma allow the quasiconformal distortion of polyhedral maps between circle
 packings to be bounded in terms of the maximum degree. We can arrange that all vertices have degree $6$ except those involved
 in the combinatorial welding. The largest degree arising when we weld using
 the hyperbolic radii of the neighbors of $v_B$ in the packing $\mathcal{P}_B$
 can be bounded in terms of the ratio between the radii of the
 largest and smallest circles adjacent to $v_B$. Since we began with a
 hex-packing of $\Omega$ with circles all of equal radius, this ratio is
 controlled in the limit by the maximum and minimum value of the derivative on
 the boundary $\partial{B}$ of the conformal map $h \circ \pi^{-1}$.

\begin{figure}[!h]
	\begin{center}
		\includegraphics[scale=0.5]{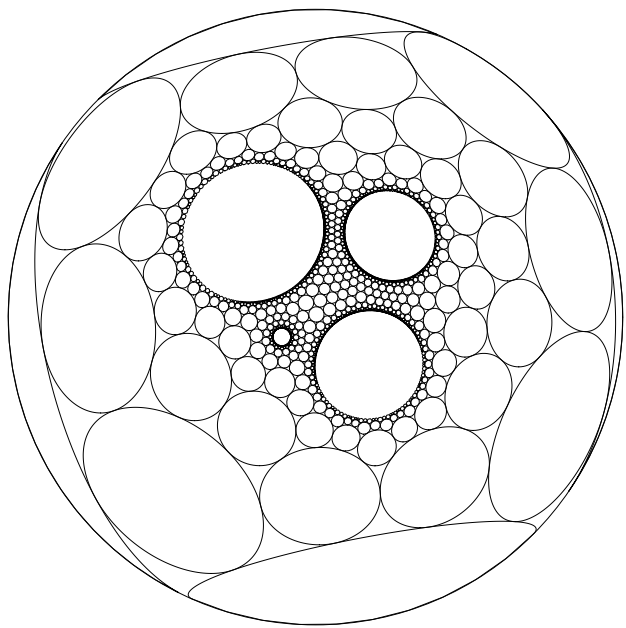}
	\end{center}
\caption{A circle packing of a circle domain with four complementary components.}
	\label{F: circle domain 2}
\end{figure}

\begin{figure}[!h]
	\begin{center}
		\includegraphics[scale=0.5]{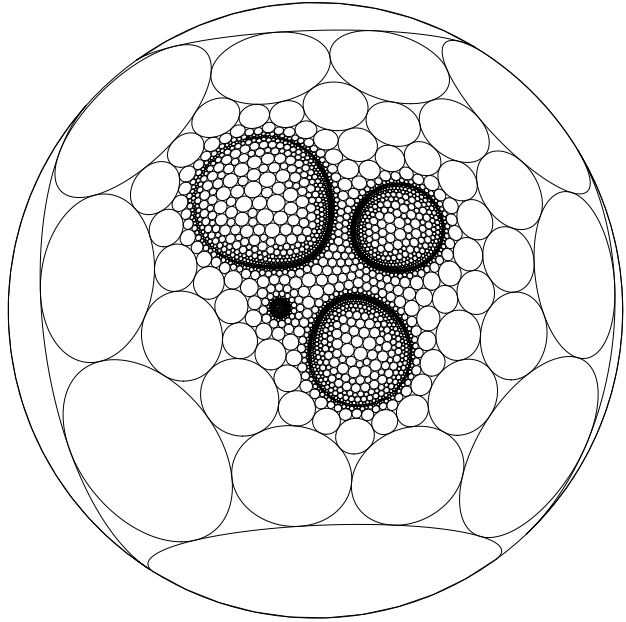}
	\end{center}
\caption{The spherical circle packing  used to compute the intrinsic circle domain conformally equivalent to the circle domain shown in Fig.~\ref{F: circle domain 2}. }
	\label{F: intrinsic 2}
\end{figure}

\subsection{Practical considerations}
Examining the figures, the reader will note that in our numerical experiments
we have not restricted ourselves to using hex packings as our starting
point. Instead we have used packings with few layers of degree $7$ vertices
near the boundary, so that there can are many very small circles on the boundary
 without needing a very large number of circles in total. This was done to reduce the errors from the combinatorial
welding step. We pay for this by having larger circles far from the boundary,
and losing the rigorous distortion bounds, but this seems a worthwhile trade since we expect the Schwarzian derivative of
the map that we are approximating to be small on this region. In
Figures~\ref{F: circle domain} and~\ref{F: intrinsic circle domain} there is
a circle of degree $15$ fixed by a rotational symmetry of the packing of order
$3$. The Schwarzian derivative of the map we are approximating must certainly
vanish at the center of this circle, by symmetry considerations. Apart from
two circles of this type, the maximum degree in the packing is $10$. In the
packings in Figures~\ref{F: circle domain 2} and~\ref{F: intrinsic 2}, the
maximum degree is $7$. 


Finally we comment that it may be more satisfactory from the point of view of
computation, and also to simplify the convergence proof, to perform the
discrete conformal welding using packings with
specified overlap angles rather than restricting ourselves to tangency
packings. However, we have not implemented this in practice.


\begin{thebibliography}{AAA}
\bibitem{CO} T.~Carroll and J.~Ortega-Cerd\`{a}, \emph{The univalent
    Bloch-Landau constant, harmonic symmetry and conformal glueing}, J.~Math.~Pures Appl. (9) {\bf{92}} (2009), no. 4, 396--406.
\bibitem{Cr} E.~Crane, \emph{A bound for Smale's mean value conjecture for
    complex polynomials}, Bull.~Lond.~Math.~Soc.~{\bf 39} (2007), no.~5, 781--791.
\bibitem{Gol} G.~M.~Goluzin, \emph{Geometric theory of functions of a complex variable}, Transl.~Math.~Monogr., Vol.~26, AMS (1969).
\bibitem{HS} Z.~X.~He and O.~Schramm, \emph{Fixed points, Koebe uniformization
    and circle packings}. Ann.~of~Math.~(2)~{\bf 137} (1993), 369--406.
\bibitem{OP} J.~Ortega-Cerd\`{a} and B.~Pridhnani, \emph{The  P\'olya-Tchebotar\"{o}v problem}, Harmonic analysis and partial
  differential equations, Contemp.~Math.~{\bf{505}} (2010), 153--170.
\bibitem{RS} B.~Rodin and  D.~Sullivan, \emph{The convergence of circle
    packings to the Riemann mapping}, J.~Differential Geom.~{\bf{26}} (1987), no. 2, 349–-360. 
\bibitem{S} O.~Schramm, \emph{Transboundary extremal length}, J.~Anal.~Math.~{\bf{66}} (1995), 307--329. 
\bibitem{St} K.~Stephenson, \emph{Introduction to circle packing: the theory
    of discrete analytic functions}, Cambridge University Press, 2005.
\bibitem{CP} K.~Stephenson et al. CirclePack, software for circle
  packing. \begin{verbatim} www.math.utk.edu/~kens/circlepack\end{verbatim} 
\bibitem{W} G.~B.~Williams Discrete conformal welding. (English summary)
Indiana Univ.~Math.~J.~53 (2004), no.~3, 765--804. 

\end{thebibliography}
\end{document}